\documentclass[12pt,a4paper]{amsart}
\usepackage[latin1]{inputenc}
\usepackage{amssymb, amsmath}
\usepackage{geometry}
\usepackage{enumerate}
\newtheorem{theorem}{Theorem}[section]
\newtheorem{definition}[theorem]{Definition}
\newtheorem{lemma}[theorem]{Lemma}
\newtheorem{corollary}[theorem]{Corollary}
\newtheorem{proposition}[theorem]{Proposition}
\newtheorem{example}[theorem]{Example}
\newtheorem{remark}[theorem]{Remark}

\begin{document}
\title[Characterization of restricted type spaces]{Characterization of the restricted \\  type spaces $R(X)$}
\author{Javier Soria}
\address{Department of  Applied Mathematics and Analysis, University of Barcelona, Gran Via 585, E-08007 Barcelona, Spain}
\email{soria@ub.edu}

\author{Pedro Tradacete}
\address{Mathematics Department, University Carlos III de Madrid, 28911 Legan\'es, Madrid, Spain.}
\email{ptradace@math.uc3m.es}

\thanks{Both authors have been partially supported by the Spanish Government Grant MTM2010-14946. The second author has been partially supported by the Spanish Government Grant MTM2012-31286 and Juan de la Cierva program}


\begin{abstract}
We study functorial properties of the spaces $R(X)$, which have been recently introduced as a central tool in the analysis of the Hardy operator minus the identity on decreasing functions. In particular, we provide conditions on a minimal Lorentz space $\Lambda_{\varphi}$ so that the equation $R(X)=\Lambda_{\varphi}$ has a solution within the category of rearrangement invariant  (r.i.) spaces. Moreover, we show that if $R(X)=\Lambda_{\varphi}$, then we can always take $X$ to be the minimal r.i.\ Banach  range space for the Hardy operator defined in $\Lambda_{\varphi}$.
\end{abstract}
\maketitle

\thispagestyle{empty}

\section{Introduction}
Let $X$ be a rearrangement invariant space (r.i.) on $(0,\infty)$, satisfying that the function $1/{(1+s)}\in X$. Associated with $X$, we can consider the space $R(X)$, introduced in \cite{Soria:10} (which appears naturally in the study of the norm of the Hardy operator minus the identity in the cone
of radially decreasing functions \cite{Boza-Soria}). This is defined as the minimal Lorentz function space $\Lambda_{W_X}$, with
\begin{equation}\label{wxeg}
W_X(t)=\bigg\|\frac1{1+\frac{\cdot}{t}}\bigg\|_X=\|E_{1/t}g\|_X,
\end{equation}
where $g(s)={1}/{(1+s)}$ and $E_t$ denotes the usual dilation operator (cf.\ \cite[\S 3]{Bennett-Sharpley}). It can be noted that the function $g(s)=1/{(1+s)}$ belongs to an r.i. space $X$ if and only if the inclusion $(L^{1,\infty}\cap L^\infty)\subset X$  holds \cite{Soria:10}.

Recall that, for a quasi-concave function $\varphi$, that is, an increasing function such that $\varphi(t)/t$ is decreasing and $\varphi(t)\neq0$ for $t>0$ (cf. \cite[p. 69]{Bennett-Sharpley}), the minimal Lorentz space $\Lambda_\varphi$ is defined as
$$
\Lambda_\varphi=\bigg\{f:\Vert f\Vert_{\Lambda_\varphi}=\int_0^\infty f^*(t)\,d\varphi(t)<\infty\bigg\},
$$
where $f^*$ is the decreasing rearrangement of $f$ \cite{Bennett-Sharpley}. Similarly, the maximal Lorentz space (or Marcinkiewicz space) $M_\varphi$ is the r.i.\ space of all measurable functions $f$ such that
$$
M_\varphi=\bigg\{f:\|f\|_{M_\varphi}=\sup_{t>0} f^{**}(t)\varphi(t)<\infty\bigg\},
$$
where
$$
f^{**}(t)=\frac1t\int_0^tf^*(s)\,ds.
$$
Another important space associated to $\varphi$ is the weak-type Lorentz space
\begin{equation}\label{laui}
\Lambda_\varphi^{1,\infty}=\bigg\{f:\Vert f\Vert_{\Lambda_\varphi^{1,\infty}}=\sup_{t>0} f^{*}(t)\varphi(t)<\infty\bigg\},
\end{equation}
and it is easy to prove that $\Lambda_\varphi\subset M_\varphi\subset\Lambda_\varphi^{1,\infty}$. Note that, in general, $\Lambda_\varphi^{1,\infty}$ is a quasi-Banach space which need not be locally convex (cf. \cite{CPSS}).

For a given r.i.\ space $X$, the fundamental function $\varphi_X$ is defined as
$$
\varphi_X(t)=\|\chi_A\|_X, \hbox{ where } |A|=t.
$$
On an r.i. space, this expression is independent of the set $A$, so $\varphi_X$ is a well-defined quasi-concave function. This allows us to consider the minimal and maximal spaces associated with $X$: $\Lambda(X)=\Lambda_{\varphi_X}$ and $M(X)=\Lambda_{\varphi_X}$. It is well known (see \cite{Bennett-Sharpley}) that for every r.i. space $X$ we have
$$
\Lambda(X)\subset X\subset M(X).
$$
Now, as noted in \cite{Rodriguez-Soria}, if the space $X$ satisfies the condition mentioned above that $g(s)=1/(1+s)$ belongs to $X$, then we can extend this chain of inclusions as follows
$$
R(X)\subset\Lambda(X)\subset X\subset M(X).
$$

It was proved in \cite{Rodriguez-Soria} that every  $W_X$ as in \eqref{wxeg} satisfies $W_X(t)\geq C t\log(1+1/{t})$, for some constant $C>0$. Our main interest is to consider a converse result, namely, whether every Lorentz function space $\Lambda_\varphi$ whose fundamental function $\varphi$ satisfies the inequality
\begin{equation}\label{filo}
\varphi(t)\geq Ct\log(1+1/{t}),
\end{equation}
can be equal to $R(X)$, for some r.i.\ $X$.

It is known that this question has a positive answer if the upper fundamental index of the space $X$ (see \cite{Bennett-Sharpley})
$$
\overline{\beta}_X=\inf_{s>1}\frac{\log\overline{\varphi}_X(s)}{\log s},
$$
where
$$
\overline{\varphi_X}(s)=\sup_{t>0}\frac{\varphi_X(st)}{\varphi_X(t)},
$$
satisfies $\overline{\beta}_X<1$. Indeed, \cite[Theorem 2.2]{Rodriguez-Soria} asserts that $\overline{\beta}_X<1$ is actually equivalent to the identity $R(X)=\Lambda(X)$. Therefore, in this case, $R$ is constant on all r.i.\ spaces having the same fundamental function $\varphi$; i.e., those $X$ for which $\Lambda_{\varphi}\subset X\subset M_{\varphi}$ \cite{Bennett-Sharpley}.

In our study of the equation $R(X)=\Lambda_\varphi$, we will elaborate first on its connection with the optimal range for the Hardy operator on $\Lambda_\varphi$ (provided such space exists). This will allow us to find a solution to the equation when the space $X$ is only an  r.i.\ quasi-Banach space. In the remaining sections we will provide conditions on a quasi-concave function $\varphi$ satisfying \eqref{filo} in order to have $\Lambda_\varphi=R(X)$, with $X$ being a Marcinkiewicz space or a Lorentz space.

The terminology used in this paper follows the monograph \cite{Bennett-Sharpley}, to which the reader is referred for further explanations concerning rearrangement invariant spaces and related concepts.

\section{Optimal range for the Hardy operator}

We recall the definition of  the Hardy operator in $\mathbb{R}^+$:
\begin{equation}\label{haop}
Sf(t)=\frac1t\int_0^tf(r)dr.
\end{equation}
A simple calculation shows that, for any $s,t>0$, then
\begin{equation*}
  S\chi_{[0,t]} (s) = \min\Big\{1,\frac{t}{s}\Big\}\approx\frac{1}{1+\frac{s}{t}}=E_{1/t}g(s),
\end{equation*}
where $g$ and $E_t$ are defined as in \eqref{wxeg}. This remark yields the following important fact.

\begin{lemma}\label{lemma R-S}
Let $\varphi$ be quasi-concave and $X$ an r.i.\ Banach space. The following are equivalent:
\begin{enumerate}[{\rm(}i{\rm\,)}]
\item{$\Lambda_\varphi\subset R(X)$}.
\item{$\|S\chi_{[0,t]}\|_X\lesssim\varphi(t)$}.
\item{$S:\Lambda_\varphi\rightarrow X$ is bounded}.
\end{enumerate}
\end{lemma}

\begin{proof}
By definition, $R(X)$ is the Lorentz space $\Lambda_{W_X}$, where $W_X(t)=\|E_{1/t}g\|_X\approx\|S\chi_{[0,t]}\|_X$. Therefore, $\Lambda_\varphi\subset \Lambda_{W_X}$ is equivalent to $\|S\chi_{[0,t]}\|_X\lesssim\varphi(t)$ \cite{CPSS}. This shows the equivalence of $(i)$ and $(ii)$.

The implication $(iii)\Rightarrow (ii)$ is immediate. Let us now see that $(ii)$ implies $(iii)$. First, notice that for a measurable set $A$, with measure $|A|$, we have that
$$
\|S\chi_A\|_X\leq\|S\chi_{[0,|A|]}\|_X\leq C\|\chi_A\|_{\Lambda_\varphi}.
$$
Now, given $f\in\Lambda_\varphi$, denote $A_n=\big\{x:2^n<|f(x)|\leq2^{n+1}\big\}$, for $n\in\mathbb Z$. Using that
$$
\|f\|_{\Lambda_\varphi}=\int_0^\infty \varphi(\lambda_f(t))\,dt,
$$
where $\lambda_f(t)=|\{x:|f(x)|>t\}$ is the distribution function of $f$ then, it follows that

\begin{align*}
\|Sf\|_X&=\bigg\|\sum_{n\in\mathbb Z} S(f\chi_{A_n})\bigg\|_X\leq C\sum_{n\in\mathbb Z} 2^{n+1}\|\chi_{A_n}\|_{\Lambda_\varphi}\\
&\le C\sum_{n\in\mathbb Z} 2^{n+1}\varphi(\lambda_f(2^n))\leq 4C\|f\|_{\Lambda_\varphi}.
\end{align*}
\end{proof}

The equivalence of conditions $(i)$ and $(iii)$ suggests that, in order to obtain an equality in $(i)$, we should consider the optimal range for the Hardy operator \eqref{haop}  on $\Lambda_\varphi$.

\begin{definition}
Given a quasi-concave function $\varphi$, let $\mathfrak{R}[S,\Lambda_\varphi]$ denote the minimal r.i.\ Banach function space $Y$ such that $S:\Lambda_\varphi\rightarrow Y$ is bounded.
\end{definition}

Note that the minimal space $\mathfrak{R}[S,\Lambda_\varphi]$ may not exist in general. However, we will see in Theorem \ref{mrri} that the existence of this space is equivalent to condition \eqref{filo}.

As far as we know, the problem of determining the optimal space $Y$ (among r.i.\ spaces) such that $S:\Lambda_\varphi\rightarrow Y$ is bounded has not been studied before. An easy duality  argument (using that $\Lambda_\varphi'=M_{\varphi_a}$) relates this problem with that of finding the optimal r.i.\ space $X$ such that $S':X\rightarrow M_{\varphi_a}$ is bounded, where $\varphi_a(t)=t/\varphi(t)$ and  $S'$ is  the conjugate Hardy operator:
$$
S'f(t)=\int_t^\infty\frac{f(s)}{s}ds.
$$

If the minimality condition we consider here is relaxed, and one looks for an optimal domain or range space among the class of all Banach lattices (or Banach function spaces), then vector measure techniques are used to characterize these cases  (see \cite{Okada-Ricker-Sanchez} and the references therein). However, note that in fact, as pointed out in \cite{Delgado-Soria}, the optimal domain for the Hardy operator is never an r.i.\ space. Similar questions, related to optimal Sobolev embeddings for r.i.\ spaces, were also considered in \cite{Edmunds-Kerman-Pick}.

In \cite{Nekvinda-Pick1} and \cite{Nekvinda-Pick2} this kind of optimal range (respectively domain) problems within the class of Banach lattices were studied for the Hardy operator and  $L^p$ spaces.

We now characterize the existence of $\mathfrak{R}[S,\Lambda_\varphi]$ and show an explicit description of its norm:

\begin{theorem}\label{mrri}
Let $\varphi$ be a quasi-concave function. Then,  $\varphi$ satisfies \eqref{filo} if and only if the space $\mathfrak{R}[S,\Lambda_\varphi]$ exists. In this case, $\mathfrak{R}[S,\Lambda_\varphi]$ coincides with the space:
$$
X=\Big\{f\in L^1+L^\infty:f^{**}\leq (Sg)^{**},\textrm{ for some decreasing }g\in\Lambda_\varphi\Big\},
$$
endowed with the norm
$$
\|f\|_X=\inf\Big\{\|g\|_{\Lambda_\varphi}: f^{**}\leq (Sg)^{**}\Big\}.
$$
\end{theorem}

\begin{proof}
Assume $\varphi$ satisfies \eqref{filo}. Let us start by proving that $\|\cdot\|_X$ actually defines a norm. It is trivial that $\|f\|_X=0$, when $f=0$ and that $\|\lambda f\|_X=|\lambda|\|f\|_X$.

Now, suppose that $\|f\|_X=0$. Then, there exists $g_n$ in $\Lambda_\varphi$, with $f^{**}\leq(Sg_n)^{**}$, such that $\|g_n\|_{\Lambda_\varphi}\rightarrow0$. Observe that, since $\varphi(t)\gtrsim t\log(1+1/t)\approx\Vert S\chi_{[0,t]}\Vert_{L^1+L^\infty}$ and using Lemma~\ref{lemma R-S}[(ii)$\Rightarrow$(iii)], we have that
$$
S:\Lambda_\varphi\rightarrow L^1+L^\infty
$$ is bounded, and hence $\|Sg_n\|_{L^1+L^\infty}\rightarrow0$. Since $f^{**}\leq(Sg_n)^{**}$ for every $n$, we have that $\|f\|_{L^1+L^\infty}\leq \inf\|Sg_n\|_{L^1+L^\infty}=0$, which shows that $f=0$.

Now, to prove the triangle inequality, take  $f_1, f_2\in X$. For each pair of decreasing functions $g_1,g_2$ in $\Lambda_\varphi$, such that $f_i^{**}\leq (Sg_i)^{**}$ for $i=1,2$, we have that

\begin{eqnarray*}
   (f_1+f_2)^{**}(t)&\leq & (Sg_1)^{**}(t)+(Sg_2)^{**}(t)=\frac1t\int_0^t\big((Sg_1)^*(s)+(Sg_2)^*(s)\big)\,ds \\
   &=& \frac1t\int_0^t S(g_1+g_2)(s)ds=(S(g_1+g_2))^{**}(t).
\end{eqnarray*}

Therefore,
$$
\|f_1+f_2\|_X\leq\|g_1+g_2\|_{\Lambda_\varphi}\leq\|g_1\|_{\Lambda_\varphi}+\|g_2\|_{\Lambda_\varphi},
$$
and since this holds for every $g_1,g_2\in\Lambda_\varphi$ such that $f_i^{**}\leq (Sg_i)^{**}$, we get that $\|f_1+f_2\|_X\leq \|f_1\|_X+\|f_2\|_X$.

Hence $\|\cdot\|_X$ defines a norm in $X$, which is clearly rearrangement invariant. Let us see now that with this norm $X$ is also complete.

Suppose $f_n$ is a sequence in $X$ with $\sum_{n=1}^\infty\|f_n\|_X<\infty$, then we want to prove that $\sum_{n=1}^\infty f_n$ converges in $X$. Splitting the sum into its positive and negative parts we can assume that $f_n$ are all positive functions.

By hypothesis, for each $n$ let $g_n$ be a decreasing function in $\Lambda_\varphi$ with
$$
f_n^{**}\leq (S g_n)^{**} \,\,\,\,\,\, \mathrm{ and }\,\,\,\,\,\,\|g_n\|_{\Lambda_\varphi}\leq\|f_n\|+2^{-n}.
$$
In particular,
$$
\sum_{n=1}^\infty \|g_n\|_{\Lambda_\varphi}<\infty,
$$
and since $\Lambda_\varphi$ is complete, then the series $\sum_{n=1}^\infty g_n$ converges in $\Lambda_\varphi$.

We claim that, for each $k>0$,
$$
\Big(\sum_{n=k}^\infty f_n\Big)^{**}\leq\Big(S\Big(\sum_{n=k}^\infty g_n\Big)\Big)^{**}.
$$
Indeed, first note that, by Lemma~\ref{lemma R-S},  $S:\Lambda_\varphi\rightarrow L^1+L^\infty$ is bounded and hence, using  \cite[\S 2 Theorem 4.6]{Bennett-Sharpley},
$$
\sum_{n=1}^\infty \|f_n\|_{L^1+L^\infty}\leq\|S\|\sum_{n=1}^\infty\|g_n\|_{\Lambda_\varphi}<\infty,
$$
and we conclude that $\sum_{n=1}^\infty f_n\in L^1+L^\infty$.

Now for fixed $k>0$, let
$$
h_{k,n}=\sum_{j=k}^{k+n}f_j.
$$
Clearly
$$
h_{k,n}\uparrow \sum_{j=k}^\infty f_j
$$
almost everywhere, so $h_{k,n}^{**}\uparrow (\sum_{j=k}^\infty f_j)^{**}$ point-wise (as $n\rightarrow\infty$.) On the other hand, for each $n\in \mathbb{N}$ we have
$$
\Big(\sum_{j=k}^{k+n} f_j\Big)^{**}\leq\sum_{j=k}^{k+n} (Sg_j)^{**}=\Big(S\Big(\sum_{j=k}^{k+n}g_j\Big)\Big)^{**}\leq\Big(S\Big(\sum_{j=k}^\infty g_j\Big)\Big)^{**}.
$$
Hence, taking the limit as $n\rightarrow\infty$ we have that
$$
\Big(\sum_{j=k}^\infty f_j\Big)^{**}\leq\Big(S\Big(\sum_{j=k}^\infty g_j\Big)\Big)^{**}
$$
as claimed.

Now, note that since
$$
\lim_{k\rightarrow\infty}\Big\|\sum_{n=k}^\infty g_n\Big\|_{\Lambda_\varphi}=0 \hspace{1cm}\textrm{and}\hspace{1cm}\Big(\sum_{n=k}^\infty f_n\Big)^{**}\leq\Big(S\Big(\sum_{n=k}^\infty g_n\Big)\Big)^{**},
$$
then, by the definition of the norm in $X$ we have that
$$
\lim_{k\rightarrow\infty}\Big\|\sum_{n=k}^\infty f_n\Big\|=0,
$$
or equivalently, that $\sum_{n=1}^\infty f_n$ converges in $X$. Therefore, $X$ is an r.i.\ Banach space.

Now, for any decreasing $f\in \Lambda_\varphi$
$$
\|Sf\|_X=\inf\{\|g\|_{\Lambda_\varphi}:(Sf)^{**}\leq(Sg)^{**}\}\leq\|f\|_{\Lambda_\varphi}.
$$
Thus, $S:\Lambda_\varphi\rightarrow X$ is bounded, which by the definition of $\mathfrak{R}[S,\Lambda_\varphi]$ means that $\mathfrak{R}[S,\Lambda_\varphi]\subset X$. For the converse inclusion, pick any r.i.\ Banach space $Y$ such that $S:\Lambda_\varphi\rightarrow Y$ is bounded. If $f\in X$ then,  for each decreasing function $g\in\Lambda_\varphi$ with $f^{**}\leq (Sg)^{**}$, we have that \cite[\S 2 Theorem 4.6]{Bennett-Sharpley}
$$
\|f\|_Y\leq\|Sg\|_Y\leq\|S\|\|g\|_{\Lambda_\varphi},
$$
which implies that $\|f\|_Y\leq\|S\|\|f\|_X$; i.e.,  $X\subset Y$. This proves the minimality condition, and hence $\mathfrak{R}[S,\Lambda_\varphi]$ coincides with $X$.

Conversely, if the minimal range space $\mathfrak{R}[S,\Lambda_\varphi]$ exists, in particular we have that
$$
S:\Lambda_\varphi\rightarrow \mathfrak{R}[S,\Lambda_\varphi]\subset L^1+L^\infty,
$$
and by Lemma~\ref{lemma R-S} we obtain that $ t\log(1+1/t) \approx\|S\chi_{[0,t]}\|_{L^1+L^\infty}\lesssim\|\chi_{[0,t]}\|_{\Lambda_\varphi}=\varphi(t)$, which gives \eqref{filo}.
\end{proof}

Note that by Lemma \ref{lemma R-S}[(iii)$\Rightarrow$(i)], it follows that, if $\varphi$ satisfies \eqref{filo},  we always have $$\Lambda_\varphi\subset R(\mathfrak{R}[S,\Lambda_\varphi]).$$ As an immediate application we get:

\begin{corollary}\label{coro optimal}
Given $\varphi$ satisfying \eqref{filo}, if there exists an r.i.\ Banach space $X$ such that $R(X)=\Lambda_\varphi$, then $R(\mathfrak{R}[S,\Lambda_\varphi])=\Lambda_\varphi$.
\end{corollary}

\begin{proof}
Let $X$ be such that $R(X)=\Lambda_\varphi$. By Lemma \ref{lemma R-S}[(i)$\Rightarrow$(iii)], we have that $S:\Lambda_\varphi\rightarrow X$ is bounded. Hence, by the minimality of $\mathfrak{R}[S,\Lambda_\varphi]$ we must have $\mathfrak{R}[S,\Lambda_\varphi]\subset X$. Therefore, by the monotonicity of the operation $R$, it holds that
$$
\Lambda_\varphi\subset R(\mathfrak{R}[S,\Lambda_\varphi])\subset R(X) =\Lambda_\varphi,
$$
as claimed.
\end{proof}

In particular, this shows that, provided $\varphi$ satisfies \eqref{filo}, the equation $R(X)=\Lambda_\varphi$ has a solution if and only if $X=\mathfrak{R}[S,\Lambda_\varphi]$ is a solution (though it may not be the only one).

\begin{corollary}
Given $\varphi$ satisfying \eqref{filo},  the following conditions are equivalent:
\begin{enumerate}
\item[{(i)}] There exists an r.i.\ Banach space $X$ such that $R(X)=\Lambda_\varphi$.
\item[{(ii)}] There exists a constant $C>0$ such that, for any $t>0$, if a decreasing function $g_t$ satisfies that
$$
\int_0^sg_t(u)\log\Big(\frac{s}{u}\Big)du\geq t\log\Big(1+\frac{s}{t}\Big),
$$
for every $s>0$, then
$$
\int_0^\infty g_t(u)d\varphi(u)\ge C\varphi(t).
$$
\item[{(iii)}] There is a constant $C>0$ such that, if $g$ is a decreasing function with
$$
\int_0^sg(u)\log\Big(\frac{s}{u}\Big)du\geq \log(1+s),
$$
for every $s>0$, then for every $t>0$ it also satisfies
$$
\int_0^\infty g\Big(\frac{u}{t}\Big)d\varphi(u)\ge C\varphi(t).
$$
\end{enumerate}
\end{corollary}

\begin{proof}
Using Corollary~\ref{coro optimal}, we know that $R(X)=\Lambda_\varphi$ has a solution if and only if $\Lambda_\varphi=R(\mathfrak{R}[S,\Lambda_\varphi])$. Moreover, we always have
$$
\Lambda_\varphi\subset R(\mathfrak{R}[S,\Lambda_\varphi]).
$$

The converse embedding is equivalent to
$$
\varphi(t)\lesssim W_{\mathfrak{R}[S,\Lambda_\varphi]}(t)=\inf\Big\{\|g\|_{\Lambda_\varphi}:(S\chi_{(0,t)})^{**}\leq(Sg)^{**}\,\textrm{with }g\downarrow\Big\}.
$$

But, a straightforward computation shows that a decreasing function $g_t$ satisfies $(S\chi_{(0,t)})^{**}\leq(Sg_t)^{**}$ if and only if
$$
\int_0^sg_t(u)\log\Big(\frac{s}{u}\Big)du\geq t\log\Big(1+\frac{s}{t}\Big).
$$

This shows the equivalence of the first two statements. The equivalence with the third one follows directly from the fact that the dilation operator $E_{t}$ commutes with the Hardy operator: $SE_t(g)=E_tS(g)$.
\end{proof}

\begin{remark}\label{laor}{\rm
It is easy to see that, under condition \eqref{filo}, we always have that $\Lambda_\varphi\subset \mathfrak{R}[S,\Lambda_\varphi]$. In fact, if $f\in \Lambda_\varphi$, then taking $g=f^*\in \Lambda_\varphi$ we get that $f^{**}=S(g)\le (Sg)^{**}$. Moreover, we can prove the following  characterization for the case of equality, in terms of the upper Boyd index $\overline{\alpha}_{\Lambda_\varphi}$ \cite[\S 3 Definition 5.12]{Bennett-Sharpley}.
}\end{remark}

\begin{corollary}\label{Rango=Lorentz}
Given $\varphi$ satisfying \eqref{filo}, we have that
$$
\Lambda_\varphi=\mathfrak{R}[S,\Lambda_\varphi]
$$
if and only if $\overline{\alpha}_{\Lambda_\varphi}<1$. If this holds true, then, in fact,  $\Lambda_\varphi=R(\mathfrak{R}[S,\Lambda_\varphi])$.
\end{corollary}

\begin{proof}
By Remark~\ref{laor}, $\Lambda_\varphi\subset \mathfrak{R}[S,\Lambda_\varphi]$, and hence equality holds if and only if $\mathfrak{R}[S,\Lambda_\varphi]\subset\Lambda_\varphi$ which, by Theorem~\ref{mrri}, is equivalent to the boundedness $S:\Lambda_\varphi\rightarrow \Lambda_\varphi$. Finally, this condition is known to be equivalent to the inequality  $\overline{\alpha}_{\Lambda_\varphi}<1$ \cite[\S 3 Theorem~5.17]{Bennett-Sharpley}.
\end{proof}

Notice that, in general, the relation between the upper fundamental index and the upper Boyd index, for an r.i.\ space $X$ with fundamental function $\varphi_X$, is given by the inequality  $\overline{\beta}_{\varphi_X}\leq\overline{\alpha}_X$ \cite[pp. 177-178]{Bennett-Sharpley}.
When $X$ is a Lorentz space, it is easy to show that, in fact,  the equality always holds (examples of non-Lorentz r.i.\ spaces with strict inequality are known \cite{Shimogaki}). This result agrees with the following remark: If $X=\Lambda_\varphi$, we know that $\overline{\beta}_\varphi<1$ is equivalent to $R(\Lambda_\varphi)=\Lambda_\varphi$ \cite[Theorem 2.2]{Rodriguez-Soria}, which implies that $R(\mathfrak{R}[S,\Lambda_\varphi])=\Lambda_\varphi$ (see Corollary~\ref{coro optimal}) and, by Corollary~\ref{Rango=Lorentz}, this is equivalent to $\overline{\alpha}_{\Lambda_\varphi}<1$. Hence, if $\overline{\alpha}_{\Lambda_\varphi}=1$, then $\overline{\beta}_\varphi=1$.

\begin{example}{\rm
Since the spaces $L^{p,1}$, $1<p<\infty$ and $L^\infty$  are all minimal Lorentz spaces, with upper Boyd index strictly less than 1, then $\mathfrak{R}[S,L^{p,1}]=L^{p,1}$ and $\mathfrak{R}[S,L^\infty]=L^\infty$. We know by Theorem \ref{mrri} that, if $p=1$, then $\mathfrak{R}[S,L^1]$ does not exist (see also Remark \ref{nolu}).
}\end{example}

We are now going to see a couple of examples for which the upper Boyd index is equal to 1:
\begin{example}{\rm
Let $\varphi(t)=\max\{1,t\}$. Then
$$
\mathfrak{R}[S,\Lambda_\varphi]=M_\psi,
$$
where $\psi(t)={t}/{\log(1+t)}$. In fact, first notice that $\Lambda_\varphi=L^1\cap L^\infty$. Now,  since  $S:L^1\rightarrow L^{1,\infty}$ and $S:L^\infty\rightarrow L^\infty$ are bounded, we immediately get that
$$
S:\Lambda_\varphi\rightarrow L^{1,\infty}\cap L^\infty
$$
is bounded, though $L^{1,\infty}\cap L^\infty$ is not a Banach space. But $L^{1,\infty}\cap L^\infty\subset M_\psi$ (this is in fact the smallest of all r.i.\ Banach spaces satisfying this embedding \cite[Proposition~3.3]{Rodriguez-Soria}), and hence we get that $\mathfrak{R}[S,\Lambda_\varphi]\subset M_\psi$.

For the converse inclusion, since $R(M_\psi)=\Lambda_\varphi$ \cite{Rodriguez-Soria}, using Corollary \ref{coro optimal} we get that $R(\mathfrak{R}[S,\Lambda_\varphi])=\Lambda_\varphi$. Thus, by the minimality of $M_\psi$ among those spaces with $R(X)\neq\{0\}$ \cite[Proposition 3.5]{Rodriguez-Soria}, it also holds that $M_\psi\subset \mathfrak{R}[S,\Lambda_\varphi]$.
}\end{example}

\begin{example}\label{ex210}{\rm
Let $\phi(t)=t\log(1+1/t)$. Then
$$
\mathfrak{R}[S,\Lambda_\phi]=L^1+L^\infty.
$$

In fact, since, by definition, $\mathfrak{R}[S,\Lambda_\phi]$ is an  r.i.\ Banach space,  it follows that
$$
\mathfrak{R}[S,\Lambda_\phi]\subset L^1+L^\infty.
$$
Let us prove the converse inclusion. To simplify the notation, set  $X=\mathfrak{R}[S,\Lambda_\phi]$ and let $\varphi_X$ denote its fundamental function. Since
$$
X\subset M_{\varphi_X}\subset L^1+L^\infty,
$$
by Lemma~\ref{lemma R-S} and \cite[Remark~2.7]{Rodriguez-Soria} we have that
$$
\Lambda_\phi\subset R(X)\subset R(M_{\varphi_X})\subset R(L^1+L^\infty)=\Lambda_\phi,
$$
so, in particular, we have that $R(M_{\varphi_X})=\Lambda_\phi$. Now,

\begin{eqnarray*}
\phi(t)&\gtrsim& W_{M_{\varphi_X}}(t)=\sup_{u>0}(E_{1/t}g)^{**}(u)\varphi_X(u) \\
     &=& \sup_{u>0}\frac{\varphi_X(u)}{u}\int_0^u\frac{1}{1+\frac{s}{t}}ds\\
     &=& t\sup_{u>0} \log(1+{u}/{t})\frac{\varphi_X(u)}{u}.\\
\end{eqnarray*}
Therefore, for any $u>0$,
$$
\varphi_X(u)\le\inf_{t>0}\frac{u\,\phi(t)}{t\log(1+u/t)}=\min\{1,u\}.
$$
Hence,
$$
\Lambda_{\min\{1,t\}}=L^1+L^\infty\subset\Lambda_{\varphi_X}\subset X,
$$

\noindent
which shows that $L^1+L^\infty\subset \mathfrak{R}[S,\Lambda_\phi]$ (see also  Lemma \ref{maximal phi tilde} and  Theorem~\ref{Marcinkiewicz} for a more general result).
}\end{example}

\section{The case of r.i.\ quasi-Banach spaces}

In the context of  r.i.\ quasi-Banach spaces, the equation $R(X)=\Lambda_\varphi$  has always a solution (provided that $\varphi$ is quasi-concave), as the following result shows.

\begin{theorem}
Let $\varphi$ be a quasi-concave function and let $\Lambda^{1,\infty}_\varphi$ be the  r.i.\ quasi-Banach space defined in \eqref{laui}. Then,
$R(\Lambda^{1,\infty}_\varphi)=\Lambda_\varphi.$
\end{theorem}

\begin{proof} A simple calculation shows that
\begin{eqnarray*}
W_{\Lambda^{1,\infty}_\varphi}(t) &=& \|E_{1/t}g\|_{\Lambda^{1,\infty}_\varphi} \approx \|S\chi_{[0,t]}\|_{\Lambda^{1,\infty}_\varphi} \\
 &=& \sup_s \frac1s\int_0^s\chi_{[0,t]}(u)\,du\,\varphi(s)=\|\chi_{[0,t]}\|_{M_\varphi}=\varphi(t).
\end{eqnarray*}
Therefore,
$$
R(\Lambda^{1,\infty}_\varphi)=\Lambda_{W_{\Lambda^{1,\infty}_\varphi}}=\Lambda_\varphi.
$$
\end{proof}

It is known \cite{Soria:98} that  $\Lambda^{1,\infty}_\varphi$ is a Banach space if and only if $\varphi$ satisfies the so called  $B_1$ condition \cite{AM}:
$$
\int_t^\infty\frac{\varphi(r)}{r^2}\,dr\le C\frac{\varphi(t)}t,
$$
which is equivalent to the boundedness of $S:\Lambda_\varphi\rightarrow \Lambda_\varphi$.   This condition is also characterized in terms of the upper Boyd index of $\Lambda_\varphi$ by means of the inequality $\overline{\alpha}_{\Lambda_\varphi}<1$ \cite[\S 3 Theorem 5.17]{Bennett-Sharpley}. Since $\overline{\beta}_\varphi\leq\overline{\alpha}_{\Lambda_\varphi}$ \cite[pp.\ 177--178]{Bennett-Sharpley}, we obtain in this case that, whenever $\varphi_X\approx\varphi$, then  $R(X)=\Lambda_{\varphi}$ \cite[Theorem 2.2]{Rodriguez-Soria}.

As in Theorem~\ref{mrri}, we can also consider the optimal  r.i.\ quasi-Banach space $X$ such that the operator $S:\Lambda_\varphi\rightarrow X$ is bounded. Let us denote this space by $\mathfrak{R}_q[S,\Lambda_\varphi]$. By definition $\mathfrak{R}_q[S,\Lambda_\varphi]\subset\mathfrak{R}[S,\Lambda_\varphi]$ provided both spaces exist.

\begin{theorem}
Let $\varphi$ be a quasi-concave function. The optimal range $\mathfrak{R}_q[S,\Lambda_\varphi]$ coincides with the space
$$
Y=\Big\{f\in L^{1,\infty}+L^\infty:\,f^*\leq Sg^*\textrm{, for some }g\in\Lambda_\varphi\Big\},
$$
endowed with the quasi-norm $\|f\|_Y=\inf\{\|g\|_{\Lambda_\varphi}:f^*\leq Sg^*\}.$
\end{theorem}

\begin{proof}
It is straightforward to check that $\|\cdot\|_Y$ defines an r.i.\ quasi-norm. Let us prove now that $Y$ is complete. To see this, first note that by Aoki-Rolewicz's theorem \cite{Kalton-Peck-Rogers} there exists $0<p<1$ such that $\|\cdot\|_Y$ is equivalent to a $p$-norm $\|\cdot\|_0$ (i.e., $\|x+y\|_0^p\leq\|x\|_0^p+\|y\|_0^p$). Now, $Y$ would be complete if we show that, for any sequence $(f_k)$ in $Y$ with $\sum_{k=1}^\infty\|f_k\|_0^p<\infty$, then the series
$$
\sum_{k=1}^\infty f_k
$$
converges in $Y$.

Thus, let $(f_k)$ in $Y$ with $\sum_{k=1}^\infty\|f_k\|_0^p<\infty$. Splitting each $f_k$ into its positive and negative parts, we can actually assume that $f_k$ is already a positive function in $Y$. Since the inclusion $Y\hookrightarrow L^{1,\infty}+L^\infty$ is continuous, we have that
$$
\Big(\sum_{k=1}^\infty\|f_k\|_{L^{1,\infty}+L^\infty}\Big)^p\le\sum_{k=1}^\infty\|f_k\|_{L^{1,\infty}+L^\infty}^p<\infty,
$$
and by completeness of $L^{1,\infty}+L^\infty$ we conclude
$$
f=\sum_{k=1}^\infty f_k\in L^{1,\infty}+L^\infty.
$$

Moreover, since $\sum_{k=1}^\infty\|f_k\|_0^p<\infty$, for each $k$ there exists $g_k\in\Lambda_\varphi$ with $f_k^*\leq S(g_k^*)$ and $\sum_{k=1}^\infty\|g_k\|^p_{\Lambda_\varphi}<\infty$. In particular, the series $\sum_{k=1}^\infty g_k$ converges in $\Lambda_\varphi$.

Let now fix $n\in \mathbb{N}$. Clearly
$$
\sum_{k=n}^{n+m} f_k\uparrow\sum_{k=n}^\infty f_k\leq f,
$$
so in particular
$$
\Big(\sum_{k=n}^\infty f_k\Big)^*\leq\liminf_{m\rightarrow\infty}\Big(\sum_{k=n}^{n+m} f_k\Big)^*.
$$

For $k\in \mathbb{N}$, let us denote
$$
c_k=\Big(\sum_{j=1}^\infty\|g_j\|_{\Lambda_\varphi}^p\Big)^{-1}\|g_k\|^p_{\Lambda_\varphi}.
$$
By \cite{Carro-Martin}, for every $m$ we have
$$
\begin{array}{ll}
 \Big(\underset{k=n}{\overset{n+m}\sum} f_k\Big)^*(3t) & \leq\underset{k=n}{\overset{n+m}\sum} \Big(f_k^*(t)+\frac1t\int_{c_kt}^tf_k^*(s)ds\Big)  \\
  &  \leq\underset{k=n}{\overset{n+m}\sum} \Big(S(g_k^*)(t)+\frac1t\int_{c_kt}^t\frac1s\int_0^sg_k^*(u)duds\Big)  \\
  &  \leq\underset{k=n}{\overset{n+m}\sum} \Big(S(g_k^*)(t)+\frac1t\int_0^t g_k^*(u)\log\big(\frac{t}{\max(c_kt,u)}\big)du\Big) \\
  &  \leq S\Big(\underset{k=n}{\overset{n+m}\sum} (1-\log(c_k))g_k^*\Big)(t)  \\
  &  \leq S\Big(\underset{k=n}{\overset{\infty}\sum} (1-\log(c_k))g_k^*\Big)(t).
\end{array}
$$
Therefore, we have
$$
\Big(\sum_{k=n}^\infty f_k\Big)^*(3t)\leq S\Big(\sum_{k=n}^\infty (1-\log(c_k))g_k^*\Big)(t),
$$
so by the definition of the norm, and taking into account that the dilation operator is bounded in $Y$,
$$
\begin{array}{ll}
\Big\|\underset{k=n}{\overset{\infty}\sum} f_k\Big\|_Y&\lesssim\bigg\|\underset{k=n}{\overset{\infty}\sum} \Big[1+\log\bigg(\frac{\sum_{j=1}^\infty\|g_j\|_{\Lambda_\varphi}^p}{\|g_k\|^p_{\Lambda_\varphi}}\bigg)\Big]g_k^*\bigg\|_{\Lambda_\varphi}\\[.6cm]
 & \lesssim \Big\|\underset{k=n}{\overset{\infty}\sum} g_k^*\Big\|_{\Lambda_\varphi}+\underset{k=n}{\overset{\infty}\sum}\|g_k\|_{\Lambda_\varphi}\log\bigg(\frac{(\sum_{j=1}^\infty\|g_j\|^p_{\Lambda_\varphi})^{1/p}}{\|g_k\|_{\Lambda_\varphi}}\bigg),\\[.5cm]
\end{array}
$$
which goes to $0$, as $n\rightarrow\infty$, since $\sum_{j=1}^\infty\|g_j\|_{\Lambda_\varphi}^p<\infty$ and $x\log(C/x)\lesssim x^p$,  if $0<x<C$ and $0<p<1$. Thus, we have seen that $Y$ is complete and hence it is an  r.i.\ quasi-Banach space.

Let us see now that the Hardy operator is bounded $S:\Lambda_\varphi\rightarrow Y$. Indeed, for any $f\in\Lambda_\varphi$:
$$
\|Sf\|_Y=\inf\{\|g\|_{\Lambda_\varphi}:(Sf)^*\leq Sg^*\}\leq\|f\|_{\Lambda_\varphi}.
$$

Now, suppose $S:\Lambda_\varphi\rightarrow X$ is bounded, and  let us see that $Y\subset X$. In fact, for $f\in {R}_q[S,\Lambda_\varphi]$, and any $g\in \Lambda_\varphi$ such that $f^*\leq Sg^*$, we have that
$$
\|f\|_X\leq\|Sg^*\|_X\leq\|S\|\|g\|_{\Lambda_\varphi}.
$$
Therefore, taking the infimum over all such $g$ we get that $\|f\|_X\leq\|S\|\|f\|_Y$.
\end{proof}

\begin{theorem}
Let $\varphi$ be quasi-concave. Then, $R(\mathfrak{R}_q[S,\Lambda_\varphi])=\Lambda_\varphi$.
\end{theorem}

\begin{proof}
For any $t>0$ we have that
$$
W_{\mathfrak{R}_q[S,\Lambda_\varphi]}(t)=\|E_{1/t}g\|_{\mathfrak{R}_q[S,\Lambda_\varphi]}=\inf\Big\{\|f\|_{\Lambda_\varphi}:(E_{1/t}g)^*\leq Sf^*\Big\}\leq\|\chi_{[0,t]}\|_{\Lambda_\varphi}=\varphi(t).
$$
Similarly, if $f\in\Lambda_\varphi$ is such that $(E_{1/t}g)^*\leq S f^*$, then $\chi_{[0,t]}^{**}\leq cf^{**}$. By \cite[\S3 Theorem 2.10]{Bennett-Sharpley}, this yields
$$
\varphi(t)=\|\chi_{[0,t]}\|_{\Lambda_\varphi}\leq c\|f\|_{\Lambda_\varphi}.
$$
Hence,
$$\varphi(t)\leq c\inf\{\|f\|_{\Lambda_\varphi}:(E_{1/t}g)^*\leq Sf^*\}=cW_{\mathfrak{R}_q[S,\Lambda_\varphi]}(t).$$
\end{proof}

Notice that, in fact, we have the following embedding.
\begin{proposition}
For a quasi-concave function $\varphi$ it holds that
$$
\mathfrak{R}_q[S,\Lambda_\varphi]\subset\Lambda_\varphi^{1,\infty}.
$$
\end{proposition}

\begin{proof}
If $f\in M_\varphi$, then
$$
\|Sf\|_{\Lambda_\varphi^{1,\infty}}=\sup_{t>0} (Sf)^*(t)\varphi(t)\leq\sup_{t>0} S(f^*)(t)\varphi(t)=\|f\|_{M_\varphi}.
$$
Therefore, $S:M_\varphi\rightarrow \Lambda_\varphi^{1,\infty}$ is bounded, and in particular since $\Lambda_\varphi\subset M_\varphi$, so is $$S:\Lambda_\varphi\rightarrow \Lambda_\varphi^{1,\infty}.$$ Hence, by definition
$$
\mathfrak{R}_q[S,\Lambda_\varphi]\subset\Lambda_\varphi^{1,\infty}.
$$
\end{proof}

\begin{remark}\label{nolu}{\rm In general
$\mathfrak{R}_q[S,\Lambda_\varphi]$ and $\mathfrak{R}[S,\Lambda_\varphi]$ may be different spaces. For example if $\varphi(t)=t$, and hence $\Lambda_\varphi=L^1$, then the range space $\mathfrak{R}[S,L^1]$ does not exist by Theorem~\ref{mrri}, while
$$
\mathfrak{R}_q[S,L^1]=\Big\{f\in L^{1,\infty}: \lim_{t\to 0^+}tf^*(t)=0\Big\}.
$$

In fact,  since $S:L^1\rightarrow L^{1,\infty}$ is bounded, then $\mathfrak{R}_q[S,L^1]\subset L^{1,\infty}$. Also, if $f\in \mathfrak{R}_q[S,L^1]$ there exists $g\in L^1$ such that $tf^*(t)\le \int_0^t g^*(s)\,ds$, and hence $tf^*(t)\to 0$, as $t\to 0^+$. Thus,
$$
\mathfrak{R}_q[S,L^1]\subset\Big\{f\in L^{1,\infty}: \lim_{t\to 0^+}tf^*(t)=0\Big\}.
$$

Let us now see that
$$
\Big\{f\in L^{1,\infty}: \lim_{t\to 0^+}tf^*(t)=0\Big\}\subset \mathfrak{R}_q[S,L^1],
$$
concluding thus the proof. To this end, take $f\in L^{1,\infty}$, $\Vert f\Vert_{L^{1,\infty}}= 1$, and satisfying that $\lim_{t\to 0^+}tf^*(t)=0$. Our goal is to find $g\in L^1$ such that $f^*\le S(g^*)$. For $0<t<1$, define $h(t)=\sup_{0<s<t}(sf^*(s))$, which is an increasing, positive function and $tf^*(t)\le h(t)$, $0<t<1$. By the definition of $h$, and the hypothesis on $f$, it is easy to see that $h(0^+)=0$. Also, without loss of generality we may assume that $h$ is absolutely continuous. Now, define
$$
g(t)=\begin{cases}h'(t),
      & \text{$0<t<1$}, \\
     0, & \text{$1\le t<\infty$}.
\end{cases}
$$
It is clear that $g\in L^1$ and, for $0<t<1$,

$$
f^*(t)\le \frac{h(t)}t=\frac1t\int_0^t g(s)\,ds\le\frac1t\int_0^t g^*(s)\,ds,
$$
and, for $1\le t<\infty$,
$$
f^*(t)\le \frac 1t\le\frac{h(1)}t=\frac1t\int_0^tg(s)\,ds\le\frac1t\int_0^t g^*(s)\,ds.
$$

Therefore, $f^*(t)\le S(g^*)(t)$, for every $t>0$.
}\end{remark}

\section{Marcinkiewicz spaces}

In this section we introduce the auxiliary function $\widetilde{\varphi}$, which will allow us to find a new approach for the study of the equation $R(X)=\Lambda_\varphi$. One of the main reasons to consider this new function is the fact that it is equivalent to the fundamental function of $\mathfrak{R}[S,\Lambda_\varphi]$ (Proposition~\ref{fundamental function optimal range}).

Let $\varphi:\mathbb{R}_+\rightarrow \mathbb{R}_+\cup\{0\}$ be a quasi-concave function. Let us consider the function
\begin{equation}\label{fitil}
\widetilde{\varphi}(t)=\inf_{r>0} \frac{t\varphi(r)}{r\log(1+{t}/{r})},
\end{equation}
which clearly satisfies $\widetilde{\varphi}(t)\lesssim{\varphi}(t).$

\begin{lemma}\label{varphi-tilde} Let $\widetilde{\varphi}$ be as in \eqref{fitil}. Then,
\
\begin{enumerate}
\item[{(i)}] $\widetilde{\varphi}(t)$ is increasing.
\item[{(ii)}]   $\widetilde{\varphi}(t)/t$ is decreasing.
\item[{(iii)}]   If there is a constant $C>0$ such that, for every $t>0$, $\varphi(t)\geq Ct\log(1+1/t)$, then $\widetilde{\varphi}(t)\geq C\min\{1,t\}$. In particular, $\widetilde{\varphi}(t)\neq0$ for $t>0$.
\end{enumerate}
\end{lemma}

\begin{proof}
$(i)$ Given $s<t$ it holds that
$$
\widetilde{\varphi}(s)=\inf_{u>0} \frac{u\varphi({s}/{u})}{\log(1+u)}\leq \inf_{u>0} \frac{u\varphi({t}/{u})}{\log(1+u)}=\widetilde{\varphi}(t),
$$
since $\varphi$ is increasing.

$(ii)$ This is a direct consequence of the fact that $\log(1+t)$ is increasing.

$(iii)$ By hypothesis, it holds that
$$
\widetilde{\varphi}(1)=\inf_{r>0}\frac{\varphi(r)}{r\log(1+1/r)}\geq C.
$$
Moreover, since $\widetilde{\varphi}(t)$ is increasing and ${\widetilde{\varphi}(t)}/{t}$ is decreasing, we have $\widetilde{\varphi}(t)\geq C$, for every $t\geq1$ and $\widetilde{\varphi}(t)\geq Ct$, for every $t\leq1$. Hence, for every $t>0$ we have
$$
\widetilde{\varphi}(t)\geq C\min\{1,t\}.
$$
\end{proof}

\begin{remark}{\rm
Notice that:
\begin{eqnarray*}
  \widetilde{\varphi}(0+) &=& \inf_{t>0}\widetilde{\varphi}(t) =\inf_{t>0}\inf_{r>0}\frac{t\varphi(r)}{r\log(1+t/r)}\\
   &=& \inf_{r>0}\frac{\varphi(r)}{r}\inf_{t>0}\frac{t}{\log(1+t/r)}\\
   &=&\inf_{r>0}\varphi(r)=\varphi(0+),
\end{eqnarray*}
while
\begin{eqnarray*}
 \lim_{t\to +\infty} \frac{\widetilde{\varphi}(t)}{t} &=& \inf_{t>0}\frac{\widetilde{\varphi}(t)}{t} =\inf_{t>0}\inf_{r>0}\frac{\varphi(r)}{r\log(1+t/r)}\\
   &=& \inf_{r>0}\frac{\varphi(r)}{r}\inf_{t>0}\frac{1}{\log(1+t/r)}=0.
\end{eqnarray*}
}\end{remark}

The function defined in \eqref{fitil}  will play a fundamental role in what follows, and will allow us to build the space $X=M_{\widetilde{\varphi}}$ as a candidate to solve the equation $R(X)=\Lambda_\varphi$ in the Banach case. We begin by showing the following embedding:

\begin{lemma}\label{contenido}
Given a quasi-concave function $\varphi:\mathbb{R}_+\rightarrow\mathbb{R}_+\cup\{0\}$ satisfying \eqref{filo}, let us consider the  function $\widetilde{\varphi}$ and the corresponding maximal Lorentz space $M_{\widetilde{\varphi}}$. Then,
$$
\Lambda_\varphi\subset R(M_{\widetilde{\varphi}}).
$$
\end{lemma}

\begin{proof}
Observe that, by Lemma \ref{varphi-tilde}, $\widetilde{\varphi}$ is a quasi-concave function. We will prove that $W_{M_{\widetilde{\varphi}}}(t)\leq \varphi(t)$, for $t>0$. Indeed,
\begin{eqnarray*}
    W_{M_{\widetilde{\varphi}}}(t) &=& \Big\|\frac{1}{1+{s}/{t}}\Big\|_{M_{\widetilde{\varphi}}}=\sup_{u>0}(E_{1/t}g)^{**}(u)\widetilde{\varphi}(u) \\
     &=& \sup_{u>0}\frac1u\int_0^u\frac{1}{1+{s}/{t}}ds \inf_{r>0} \frac{u\varphi(r)}{r\log(1+{u}/{r})}\\
     &=& \sup_{u>0} t\log(1+{u}/{t})\inf_{r>0} \frac{\varphi(r)}{r\log(1+{u}/{r})}\\
     &\leq& \sup_{u>0} \frac{\varphi(t) t\log(1+{u}/{t})}{t\log(1+{u}/{t})}=\varphi(t).
  \end{eqnarray*}
\end{proof}

Moreover, the function $\widetilde{\varphi}$ has the following maximal property.

\begin{lemma}\label{maximal phi tilde}
Let $\varphi$ satisfy \eqref{filo}. If for some $\phi$ we have $\Lambda_\varphi\subset R(M_\phi)$, then $\phi\lesssim\widetilde{\varphi}$, and hence
$$
\Lambda_\varphi\subset R(M_{\widetilde\varphi})\subset R(M_\phi)
$$
In other words, $\widetilde{\varphi}$ is maximal among the set of quasi-concave functions $\phi$ satisfying
$$
\sup_u\frac{t}{u}\log\Big(1+\frac{u}{t}\Big)\phi(u)\lesssim\varphi(t).
$$
\end{lemma}

\begin{proof}
Suppose $\Lambda_\varphi\subset R(M_\phi)$. Then, for every $t>0$ we have that
\begin{eqnarray*}
\varphi(t)&\gtrsim& W_{M_\phi}(t)=\sup_{u>0}(E_{1/t}g)^{**}(u)\phi(u) \\
     &=& \sup_{u>0}\frac{\phi(u)}{u}\int_0^u\frac{1}{1+ {s}/{t}}ds\\
     &=& \sup_{u>0} t\log\Big(1+\frac{u}{t}\Big)\frac{\phi(u)}{u}.\\
\end{eqnarray*}

Hence, for every $u,t>0$ we have $$\phi(u)\lesssim\frac{u\varphi(t)}{t\log(1+\frac{u}{t})},$$ and, taking the infimum over $t>0$, we conclude that $\phi(u)\lesssim\widetilde{\varphi}(u)$, as claimed.
\end{proof}

\begin{theorem}\label{Marcinkiewicz}
Let $\varphi:\mathbb{R}_+\rightarrow\mathbb{R}_+\cup\{0\}$ be a quasi-concave function satisfying \eqref{filo}. The following statements are equivalent:
\begin{enumerate}
\item[{(i)}] There exists a quasi-concave function $\phi$, such that $\Lambda_\varphi=R(M_\phi)$.
\item[{(ii)}] $\Lambda_\varphi=R(M_{\widetilde{\varphi}})$.
\item[{(iii)}] There exists $K>0$ such that for all $t>0$, there is $u_t>0$ satisfying
$$
\inf_{r>0} \frac{\varphi(r)}{r\log(1+ {u_t}/{r})}\leq \frac{\varphi(t)}{t\log(1+ {u_t}/{t})} \leq K \inf_{r>0} \frac{\varphi(r)}{r\log(1+ {u_t}/{r})}.
$$
\item[{(iv)}] There exist sequences of positive real numbers $(a_k)$, $(b_k)$ such that
$$
\varphi(t)\approx t\sup_k b_k \log\Big(1+\frac{a_k}{t}\Big).
$$
\end{enumerate}
\end{theorem}

\begin{proof}
$(i)\Leftrightarrow (ii)$: Let us suppose first that $\Lambda_\varphi=R(M_\phi)$ for some $\phi$. Then, by Lemma \ref{maximal phi tilde}, it follows that $\phi\lesssim\widetilde{\varphi}$. So we have that $M_{\widetilde{\varphi}}\subset M_\phi$. Now, this fact, together with Lemma \ref{contenido} yield
$$\Lambda_\varphi\subset R(M_{\widetilde{\varphi}})\subset R(M_\phi).$$ Since by hypothesis $\Lambda_\varphi=R(M_\phi)$, we must also have $\Lambda_\varphi=R(M_{\widetilde{\varphi}})$. This proves the implication $(i)\Rightarrow (ii)$. Since the converse is immediate, both are equivalent.

$(ii)\Leftrightarrow (iii)$: By Lemma \ref{contenido}, we have $\Lambda_\varphi=R(M_{\widetilde{\varphi}})$ if and only if there is $K>0$ such that $\varphi(t)\leq K W_{M_{\widetilde{\varphi}}}(t)$, for every $t>0$. This means that
$$
\varphi(t)\leq K \sup_{u>0} t\log\Big(1+\frac{u}{t}\Big)\inf_{r>0} \frac{\varphi(r)}{r\log(1+ {u}/{r})},
$$
which is equivalent to $(iii)$.

$(ii)\Leftrightarrow (iv)$: Suppose first that there exist sequences of positive real numbers $(a_k)$, $(b_k)$ such that
$$
\varphi(t)\approx t\sup_k b_k \log\Big(1+\frac{a_k}{t}\Big).
$$
Since
$$
W_{M_{\widetilde{\varphi}}}(t)\approx\sup_u t\log\Big(1+\frac{u}{t}\Big)\inf_r\frac{1}{\log(1+ {u}/{r})}\sup_k b_k\log\Big(1+\frac{a_k}{r}\Big),
$$
in particular, for $u=a_j$, we have
$$
W_{M_{\widetilde{\varphi}}}(t)\gtrsim t\log\Big(1+\frac{a_j}{t}\Big)\inf_r\frac{1}{\log(1+ {u}/{r})}\sup_k b_k\log\Big(1+\frac{a_k}{r}\Big)\geq b_j t\log\Big(1+\frac{a_j}{t}\Big).
$$
This holds for every $j$, so we also get that $W_{M_{\widetilde{\varphi}}}(t)\gtrsim\varphi(t)$. Since the reverse inequality holds by Lemma \ref{contenido}, we have that $\Lambda_\varphi=R(M_{\widetilde{\varphi}})$.

Conversely, if $\varphi(t)\approx W_{M_{\widetilde{\varphi}}}(t)$ then
$$
\varphi(t)\approx\sup_u \frac{t}{u}\log\Big(1+\frac{u}{t}\Big)\widetilde{\varphi}(u).
$$
Now, since $\widetilde{\varphi}$ is quasi-concave, by \cite[Proposition 3.2.6]{BK} (see also \cite{KMP}) there is an increasing sequence $(t_k)_{k\in\mathbb{Z}}$ of positive numbers such that $\widetilde{\varphi}(t_{2k+2})\approx\widetilde{\varphi}(t_{2k+1})$, ${\widetilde{\varphi}/(t_{2k})}{t_{2k}}\approx{\widetilde{\varphi}/(t_{2k+1})}{t_{2k+1}}$ and
$$\widetilde{\varphi}(u)\approx\sup_k\widetilde{\varphi}(t_{2k+1})\min\Big(1,\frac{u}{t_{2k+1}}\Big).$$ In particular, we have
\begin{eqnarray*}
  \varphi(t) &\approx& \sup_u\frac{t}{u}\log\Big(1+\frac{u}{t}\Big)\sup_k \widetilde{\varphi}(t_{2k+1})\min\Big(1,\frac{u}{t_{2k+1}}\Big)\\
   &=& t\sup_k \widetilde{\varphi}(t_{2k+1})\max\bigg\{\sup_{u\leq t_{2k+1}}\frac{\log(1+ {u}/{t})}{t_{2k+1}},\sup_{u>t_{2k+1}}\frac{\log(1+{u}/{t})}{u}\bigg\}\\
   &=& t\sup_k \frac{\widetilde{\varphi}(t_{2k+1})}{t_{2k+1}}\log\Big(1+\frac{t_{2k+1}}{t}\Big).
\end{eqnarray*}
\end{proof}

Notice also that in Theorem \ref{Marcinkiewicz}, the best constant $K$ appearing in $(iii)$ coincides with the best norm of the isomorphism between $\Lambda_\varphi$ and $R(M_{\widetilde{\varphi}})$.

\medskip

We prove next a very important feature of the function $\widetilde{\varphi}$, namely that it coincides with the fundamental function of the optimal range $\mathfrak{R}[S,\Lambda_\varphi]$.

\begin{proposition}\label{fundamental function optimal range}
Given $\varphi$ satisfying \eqref{filo}, we have that
$$\varphi_{\mathfrak{R}[S,\Lambda_\varphi]}(t)=\|\chi_{(0,t)}\|_{\mathfrak{R}[S,\Lambda_\varphi]}\approx\widetilde{\varphi}(t).$$
\end{proposition}

\begin{proof}
First, note that by Lemma \ref{contenido} we have $\Lambda_\varphi\subset R(M_{\widetilde{\varphi}})$. Hence, by Lemma \ref{lemma R-S} we have that the operator $$S:\Lambda_\varphi\rightarrow M_{\widetilde{\varphi}}$$ is bounded. Therefore, we must have $$\mathfrak{R}[S,\Lambda_\varphi]\subset M_{\widetilde{\varphi}},$$ which implies that $\varphi_{\mathfrak{R}[S,\Lambda_\varphi]}(t)\gtrsim\widetilde{\varphi}(t).$

Now, let $\psi(t)$ denote $\varphi_{\mathfrak{R}[S,\Lambda_\varphi]}(t)$. Since $\varphi$ satisfies \eqref{filo}, we have that $$\mathfrak{R}[S,\Lambda_\varphi]\subset M_\psi.$$ Hence, it follows that $$\Lambda_\varphi\subset R(\mathfrak{R}[S,\Lambda_\varphi])\subset R(M_\psi).$$ Therefore, by Lemma \ref{maximal phi tilde}, we have that $\widetilde{\varphi}(t)\gtrsim \psi(t)=\varphi_{\mathfrak{R}[S,\Lambda_\varphi]}(t)$, as claimed.
\end{proof}

We know \cite[Theorem 2.2]{Rodriguez-Soria} that in the case when $\overline{\beta}_\varphi<1$, then $R(X)=\Lambda_\varphi$ for every r.i.\ space with fundamental function equivalent to $\varphi$. In particular, we have $R(M_\varphi)=\Lambda_\varphi$, so by Theorem \ref{Marcinkiewicz} we also have $R(M_{\widetilde{\varphi}})=\Lambda_\varphi$. We will see now that, in fact in this case, $\widetilde{\varphi}\approx\varphi$.

Recall that given a quasi-concave function $\varphi$ we define \cite{Bennett-Sharpley}
$$
\overline{\varphi}(t)=\sup_{s>0}\frac{\varphi(ts)}{\varphi(s)},
$$
which is a submultiplicative function (it is actually the smallest submultiplicative function larger than $\varphi$).

\begin{lemma}\label{lema beta<1}
Let $\varphi$ be a quasi-concave function satisfying \eqref{filo}. Then,
\begin{enumerate}
\item[$(i)$] $\overline{\widetilde{\varphi}}\leq\overline{\varphi}$.
\item[$(ii)$] If $\overline{\beta}_\varphi<1$ we have that  $\overline{\beta}_{\widetilde{\varphi}}<1$.
\item[$(iii)$] $\varphi(t)\approx\widetilde{\varphi}(t)$ holds if and only if $$\overline{\varphi}(t)\lesssim\frac{t}{\log(1+t)}.$$
\end{enumerate}
\end{lemma}

\begin{proof}
$(i)$ By definition, for every $s>0$ we have
\begin{eqnarray*}
\overline{\widetilde{\varphi}}(s) &=& \sup_{t>0}\frac{\widetilde{\varphi}(ts)}{\widetilde{\varphi}(t)} \\
&=& \sup_{t>0}\sup_{u>0}\frac{u\log(1+t/u)}{t\varphi(u)}\inf_{r>0}\frac{ts\varphi(r)}{r\log(1+ts/r)}\\
&\leq& \sup_{u>0} \frac{\varphi(us)}{\varphi(u)}=\overline{\varphi}(s),
\end{eqnarray*}
where we just picked $r=us$ to get the last inequality.

$(ii)$ Follows immediately from $(i)$.

$(iii)$ Since $\widetilde{\varphi}(t)\lesssim\varphi(t)$, then  the equivalence of these two functions holds if and only if $$\varphi(t)\lesssim\widetilde{\varphi}(t)=\inf_{s>0}\frac{s\varphi(t/s)}{\log(1+s)},$$ for every $t>0$. This is the same as
$$
\frac{\varphi(t)}{\varphi(t/s)}\lesssim\frac{s}{\log(1+s)},
$$
for every $s,t>0$. Equivalently, this means that $$\overline{\varphi}(s)=\sup_{r>0}\frac{\varphi(rs)}{\varphi(r)}=\sup_{t>0}\frac{\varphi(t)}{\varphi(t/s)}\lesssim\frac{s}{\log(1+s)}.$$
\end{proof}

\begin{theorem}\label{equivalence varphi-varphitilde}
Let $\varphi$ be a quasi-concave function  satisfying \eqref{filo}. We have that $\overline{\beta}_\varphi<1$ if and only if $\varphi(t)\approx\widetilde{\varphi}(t)$.
\end{theorem}

\begin{proof}
First, let us suppose that $\overline{\beta}_{\varphi}<1$. Hence, by \cite[Theorem 2.2]{Rodriguez-Soria}
\begin{equation}\label{1}
R(M_\varphi)=\Lambda_\varphi.
\end{equation}
In particular, by Theorem \ref{Marcinkiewicz} we also have
\begin{equation}\label{2}
R(M_{\widetilde{\varphi}})=\Lambda_\varphi.
\end{equation}
On the other hand, by Lemma \ref{lema beta<1}~(ii) it holds that $\overline{\beta}_{\widetilde{\varphi}}<1$, which, by \cite[Theorem 2.2]{Rodriguez-Soria} implies that
\begin{equation}\label{3}
R(M_{\widetilde{\varphi}})=\Lambda_{\widetilde{\varphi}}.
\end{equation}
Now, putting together \eqref{1}-\eqref{3} we get that $\Lambda_\varphi=\Lambda_{\widetilde{\varphi}},$ which   is equivalent to $\varphi(t)\approx\widetilde{\varphi}(t)$.

Conversely, if the equivalence $\varphi(t)\approx\widetilde{\varphi}(t)$ holds, then by Lemma \ref{lema beta<1}~(iii) we now that
$$
\overline{\varphi}(t)\lesssim\frac{t}{\log(1+t)},
$$
for every $t>0$. Let us consider $a>1$ large enough so that ${\overline{\varphi}(a)}<a$. We have that

\begin{eqnarray*}
\int_1^\infty\overline{\varphi}(s)\frac{ds}{s^2} &=& \sum_{n=0}^\infty\int_{a^n}^{a^{n+1}}\frac{\overline{\varphi}(s)}{s^2}ds\\
&=&\sum_{n=0}^\infty\int_1^a\frac{\overline{\varphi}(a^nv)}{a^nv^2}dv\\
&\leq&\sum_{n=0}^\infty\Big(\frac{\overline{\varphi}(a)}{a}\Big)^n\int_1^a\frac{\overline{\varphi}(v)}{v^2}dv.
\end{eqnarray*}
Since ${\overline{\varphi}(a)}<a$, this is a convergent series, and using \cite[\S 3 Lemma 5.9]{Bennett-Sharpley} we conclude that $\overline{\beta}_\varphi<1$.
\end{proof}

Note that for a quasi-concave function $\varphi(t)$ satisfying \eqref{filo}, then:
\begin{equation}\label{eqca}
\min\{1,t\}\lesssim \widetilde{\varphi}(t)\lesssim\frac{t}{\log(1+t)}.
\end{equation}
In the following results we study the    equality cases in \eqref{eqca}, and prove  some important properties of the solution $R(X)=\Lambda_{\varphi}$ for the corresponding spaces:

\begin{proposition}\label{4.9}
The equivalence $$\widetilde{\varphi}(t)\approx\frac{t}{\log(1+t)}$$ holds if and only if $\varphi(t)\approx\max\{1,t\}$. Moreover, if $\psi(t)= {t}/{\log(1+t)}$,  then the Marcinkiewicz space $M_\psi$ is  minimal among the r.i.\ Banach spaces $X$   satisfying that $R(X)=\Lambda_{\max\{1,t\}}=L^1\cap L^\infty$.
\end{proposition}

\begin{proof}
It is easy to see that if  $\varphi(t)=\max\{1,t\}$, then  the equivalence $\widetilde{\varphi}(t)\approx \psi(t)$ holds. Let us now prove the converse result. We have that  $\widetilde{\varphi}\approx\psi$ if and only if there is some constant $C>0$ such that $\widetilde{\varphi}\geq C\psi(t)$. This means that
\begin{eqnarray*}
  C &\leq& \frac{\log(1+t)}{t}\inf_r \frac{t\varphi(r)}{r\log(1+ {t}/{r})} \\
   &=& \inf_r\frac{\varphi(r)}{r}\inf_t\frac{\log(1+t)}{\log(1+ {t}/{r})}\\
   &=& \inf_r\frac{\varphi(r)}{r}\min\{1,r\}\\
   &=& \min\Big\{\inf_{r\leq1}\varphi(r),\inf_{r>1}\frac{\varphi(r)}{r}\Big\}\\
   &=& \min\Big\{\lim_{r\rightarrow0}\varphi(r),\lim_{r\rightarrow\infty}\frac{\varphi(r)}{r}\Big\},
\end{eqnarray*}
and we get that $\max\{1,t\}\lesssim\varphi(t)$. The converse inequality is always true for a quasi-concave function.

Suppose now that $X$ satisfies that $R(X)=L^1\cap L^\infty$. Then $R(X)\neq\{0\}$ and, by the minimality of $M_\psi$ \cite[Proposition 3.5]{Rodriguez-Soria} among the r.i.\ Banach spaces with this property, we   have that $M_\psi\subset X$ as claimed.
\end{proof}

\begin{proposition}\label{4.10}
If $\varphi(t)\approx t\log(1+1/t)$, then $\widetilde{\varphi}(t)\approx\min\{1,t\}$ and, moreover,  $L^1+L^\infty$ is the unique r.i.\ Banach space $X$ such that $R(X)=\Lambda_\varphi$.
\end{proposition}

\begin{proof}
That $\widetilde{\varphi}(t)\approx\min\{1,t\}$ is an easy calculation. Now, recall that we have already seen in Example~\ref{ex210} that $R(L^1+L^\infty)=\Lambda_\varphi$. To finish, suppose $X$ is an r.i.\ Banach space such that $R(X)=\Lambda_\varphi$. Clearly, we have that $X\subset L^1+L^\infty$. Moreover, let $\varphi_X$ denote the fundamental function of this space. Since $X\subset M_{\varphi_X}$ we have that
$$
\Lambda_\varphi=R(X)\subset R(M_{\varphi_X}).
$$
Hence, by  Lemma~\ref{maximal phi tilde}, we conclude that
$$
M_{\widetilde{\varphi}}\subset M_{\varphi_X}\subset L^1+L^\infty.
$$
Since $\widetilde{\varphi}(t)=\min\{1,t\}$, it follows that $M_{\widetilde{\varphi}}=M_{\varphi_X}=L^1+L^\infty$ and, therefore, ${\widetilde{\varphi}}\approx \varphi_X$. But,
$$
L^1+L^\infty=\Lambda_{\min\{1,t\}}=\Lambda_{\varphi_X}\subset X\subset M_{\varphi_X}=M_{\min\{1,t\}}=L^1+L^\infty.
$$
\end{proof}

\begin{example}{\rm
We have seen in Theorem~\ref{equivalence varphi-varphitilde} that if $\overline{\beta}_\varphi<1$, then $R(M_{\widetilde{\varphi}})=\Lambda_\varphi$.   Propositions~\ref{4.9} and \ref{4.10} show that this also holds for particular choices of $\varphi$ with $\overline{\beta}_\varphi=1.$ Let us see one further example. If $\psi(t)= {t}/{\log(1+t)}$, then $\overline{\beta}_\psi=1$ and
$$
{\widetilde{\psi}}(t)=\frac{t}{\sup_{r>0}\log(1+t/r)\log(1+r)}.
$$
We observe that the function $f_r(t)=\log(1+t/r)\log(1+r)$ satisfies that $f_t(r)=f_t(t/r)$, and hence the supremum is attained when $r=t/r$; i.e., $r=\sqrt{t}$. Therefore,
$$
{\widetilde{\psi}}(t)=\frac{t}{\log^2(1+\sqrt{t})}.
$$
An easy calculation now shows that
$$
W_{M_{\widetilde{\psi}}}(t)\approx \psi(t),
$$
and hence $R(M_{\widetilde{\psi}})=\Lambda_\psi$.
}\end{example}

We are going to analyze another approach in order to study the validity of the equation $R(M_{\widetilde{\varphi}})=\Lambda_\varphi$. First, we recall that there is a canonical involution in the cone of quasi-concave functions so that, for each such $\varphi$, we can consider $\varphi^{\,i}$ defined by
$$
\varphi^{\,i}(t)=\frac{1}{\varphi(1/t)},
$$
which is also quasi-concave. We now set
\begin{equation}\label{trfi}
\varphi^{\vartriangle}(t)=\widetilde{\varphi}^{\,i}(t).
\end{equation}

\begin{theorem}
Let $\varphi$ a quasi-concave function satisfying \eqref{filo} and let $\varphi^{\vartriangle}$ be as in \eqref{trfi}. Then,
\begin{enumerate}
\item[{(i)}]
$\varphi^{\vartriangle}$ satisfies \eqref{filo}.\\ [-3mm]
\item[{(ii)}]
$\varphi^{\vartriangle\vartriangle}(t)=W_{M_{\widetilde{\varphi}}}(t)$. \\ [-3mm]
\item[{(iii)}]
$\varphi^{\vartriangle\vartriangle\vartriangle}(t)=\varphi^\vartriangle(t).$\\ [-3mm]
\item[{(iv)}]
$R(M_{\widetilde{\varphi}})=\Lambda_\varphi$ if and only if $\varphi^{\vartriangle\vartriangle}(t)\approx\varphi(t)$.
\item[{(v)}]
$\varphi^{\vartriangle}\lesssim\overline{\varphi}$.
\end{enumerate}
\end{theorem}

\begin{proof}
We have seen in \eqref{eqca} that if $\varphi(t)\gtrsim t\log(1+1/t)$, then $\widetilde{\varphi}(t)$ is quasi-concave and $\widetilde{\varphi}(t)\lesssim {t}/{\log(1+t)}$, from where it follows that $\varphi^\vartriangle(t)\gtrsim t\log(1+1/t)$, which is $(i)$.

We now prove $(ii)$:

\begin{eqnarray*}
  \varphi^{\vartriangle\vartriangle}(t) &=& \frac{1}{\displaystyle\inf_r\frac{\varphi^\vartriangle(r)}{tr\log(1+\frac{1}{tr})}}=\sup_r\frac{tr\log(1+\frac{1}{tr})}{\varphi^\vartriangle(r)} \\
   &=& \sup_r\frac{tr\log(1+\frac{1}{tr})}{\displaystyle\sup_s\frac{rs\log(1+\frac{1}{rs})}{\varphi(s)}}\\
   &=& \sup_r t\log\Big(1+\frac{1}{tr}\Big)\inf_s\frac{\varphi(s)}{s\log(1+\frac{1}{rs})}\\
   &=& W_{M_{\widetilde{\varphi}}}(t).
\end{eqnarray*}

Let us prove $(iii)$: By Lemma~\ref{contenido} we have that  $W_{M_{\widetilde{\varphi}}}(t)\leq\varphi(t)$, and hence, using $(ii)$ and the involution property:
$$
\varphi^{\vartriangle\vartriangle\vartriangle}(t)=(W_{M_{\widetilde{\varphi}}})^\vartriangle(t)\geq\varphi^\vartriangle(t).
$$

For the converse inequality, we apply again Lemma \ref{contenido}, but with the function $\varphi^\vartriangle$; that is;
$$
\varphi^{\vartriangle\vartriangle\vartriangle}(t)=W_{M_{\widetilde{\varphi^\vartriangle}}}(t)\leq\varphi^\vartriangle(t).
$$

Finally, $R(M_{\widetilde{\varphi}})=\Lambda_\varphi$ if and only if $W_{M_{\widetilde{\varphi}}}\approx\varphi$, and $(iv)$ follows from $(ii)$.

For the proof of $(v)$, note that
\begin{equation*}
\overline{\varphi}(t)=\sup_{s>0}\frac{\varphi(ts)}{\varphi(s)}\gtrsim\sup_{s>0}\frac{ts\log\Big(1+\displaystyle\frac{1}{ts}\Big)}{\varphi(s)}=\frac{1}{\underset{s>0}{\inf}\displaystyle\frac{\varphi(s)}{ts\log\Big(1+\displaystyle\frac{1}{ts}\Big)}}=\frac{1}{\widetilde{\varphi}(1/t)}.
\end{equation*}

\end{proof}

\begin{corollary}
If $\varphi_1$ and $\varphi_2$ satisfy \eqref{filo}, then
$$
R(M_{\widetilde{\varphi_1}})=R(M_{\widetilde{\varphi_2}})\Leftrightarrow
W_{M_{\widetilde{\varphi_1}}}(t)\approx W_{M_{\widetilde{\varphi_2}}}(t)\Leftrightarrow \widetilde{\varphi_1}(t)\approx\widetilde{\varphi_2}(t).
$$
\end{corollary}

\section{Lorentz spaces}

In this section we study under which conditions we have that, given a quasi-concave function $\varphi$ satisfying \eqref{filo}, there exists a Lorentz space $\Lambda_\psi$ such that $\Lambda_\varphi=R(\Lambda_\psi)$.

\begin{theorem}\label{lorf}
Let $\varphi$ be a quasi-concave function satisfying \eqref{filo}, and  let $\widehat{\varphi}(t)=(\varphi^j)'(1/t)$, where $\varphi^j(t)=t\varphi(1/t)$. If $\widehat{\varphi}$ is quasi-concave and $\underset{s\rightarrow\infty}{\lim}\varphi(s)/s=0$, then
$$
R(\Lambda_{\widehat{\varphi}})=\Lambda_\varphi.
$$
\end{theorem}

\begin{proof}

Let us start by computing the fundamental function of $R(\Lambda_\psi)$, for a given quasi-concave function $\psi$:
\begin{eqnarray*}
   W_{\Lambda_\psi}(t) &=& \int_0^\infty\psi(\mu_{E_{1/t}g}(u))\,du=\int_0^1\psi\Big(t\Big(\frac1u-1\Big)\Big)\,du\\
   &=&t\int_0^\infty\frac{\psi(r)}{(t+r)^2}\,dr=t\int_0^t\frac{\psi(r)}{(t+r)^2}\,dr+t\int_t^\infty\frac{\psi(r)}{(t+r)^2}\,dr \\
     &\approx&\frac1t\int_0^t\psi(r)\, dr+t\int_0^{1/t}\psi\Big(\frac1r\Big)\,dr.
\end{eqnarray*}

Now, if we denote $\psi_1(t)=\displaystyle \frac1t\int_0^t\psi(r)\,dr$ and $\psi_2(t)=\displaystyle t\int_0^{1/t}\psi(1/r)\,dr$, since $\psi$ is increasing, it is clear that $$
\psi_1(t)\leq\psi(t)\leq\psi_2(t).
$$
Therefore, it follows that
$$
W_{\Lambda_\psi}(t)\approx t\int_0^{1/t}\psi\Big(\frac1r\Big)\,dr.
$$

Now, since  $\widehat{\varphi}$ is quasi-concave, then
$$
W_{\Lambda_{\widehat{\varphi}}}(t)\approx t\int_0^{1/t}\widehat{\varphi}\Big(\frac1r\Big)dr=\varphi(t)-t\lim_{s\rightarrow0}s\varphi(1/s)=\varphi(t).
$$
\end{proof}

\begin{example}{\rm
If $\overline{\beta}_\varphi<1$, then we know that $R(\Lambda_\varphi)=\Lambda_\varphi$. If, for example, we take $\varphi(t)=t\log(1+1/t)$, for which $\overline{\beta}_\varphi=1$, then $\hat\varphi(t)=t/(t+1)\approx\min\{1,t\}$, which satisfies the hypothesis of Theorem~\ref{lorf}. Hence $\Lambda_{\widehat{\varphi}}=L^1+L^\infty$ and
$$
R(\Lambda_{\widehat{\varphi}})=\Lambda_{t\log(1+1/t)}.
$$
}\end{example}

\begin{theorem}
Let $\varphi$ be a quasi-concave function satisfying \eqref{filo}. Then,
$\Lambda_\varphi=R(\Lambda_\phi)$ for some quasi-concave function $\phi$ if and only if there exists an increasing sequence $(a_k)_{k\in\mathbb{Z}}$  and a decreasing sequence $(b_k)_{k\in\mathbb{Z}}$, of non-negative real numbers,  and a constant $c\geq0$ such that $$\varphi(t)\approx c+t\sum_{k\in\mathbb{Z}} b_k \log\Big(1+\frac{a_k}{t}\Big).$$
\end{theorem}

\begin{proof}
Suppose first, that $\Lambda_\varphi=R(\Lambda_\phi)$, for some $\phi$. By \cite[Proposition 3.2.6]{BK} (see also \cite{KMP}), $\phi$ is equivalent to a   function of the form
$$
{\Phi}(t)=\sum_{k\in\mathbb{Z}}\phi(t_k)\min\Big(1,\frac{t}{t_k}\Big)+\lim_{s\rightarrow0}\phi(s)+\lim_{s\rightarrow\infty}\frac{\phi(s)}{s}t,
$$
where $(t_k)_{k\in\mathbb{Z}}$ is a sequence in $(0,\infty)$ with $\lim_{k\rightarrow-\infty}t_k=0$ and $\lim_{k\rightarrow+\infty}t_k=+\infty$. Moreover, we can assume that $\lim_{s\rightarrow\infty}{\phi(s)}/{s}=0$, since otherwise we can write $\Lambda_\phi=\Lambda_{\phi_0}\cap L_1$ and
$$
\Lambda_\varphi=R(\Lambda_{\phi_0}\cap L_1)=R(\Lambda_{\phi_0})\cap R(L_1)=0.
$$

Now
$$
\begin{array}{ccl}
 W_{\Lambda_\phi}(t) & \approx & \displaystyle t\int_0^{1/t}{\Phi}\Big(\frac1r\Big)\,dr \\
   & = & t\displaystyle\sum_{k\in\mathbb{Z}}\phi(t_k)\int_0^{1/t}\min\Big(1,\frac{1}{rt_k}\Big)\,dr + t\int_0^{1/t}\underset{s\rightarrow0}{\lim}\phi(s)\,dr \\
   & \approx & t\displaystyle\sum_{k\in\mathbb{Z}}\frac{\phi(t_k)}{t_k}\log\Big(1+\frac{t_k}{t}\Big)+\phi(0^+).
\end{array}
$$

Hence, we can take $a_k=t_k$, $b_k={\phi(t_k)}/{t_k}$ and $c=\phi(0^+)$ so that
$$
\varphi(t)\approx c+t\sum_{k\in\mathbb{Z}} b_k \log\Big(1+\frac{a_k}{t}\Big).
$$

For the converse, assume now that such sequences exist so that
$$
\varphi(t)\approx c+t\sum_{k\in\mathbb{Z}} b_k \log\Big(1+\frac{a_k}{t}\Big).
$$
Let
$$
\phi(t)=c+\sum_{k\in\mathbb{Z}}b_ka_k\min\Big(1,\frac{t}{a_k}\Big).
$$
Now, it is straightforward to check that $\Lambda_\varphi=R(\Lambda_\phi)$.
\end{proof}

\end{document}